\RequirePackage{fix-cm}
\documentclass[smallcondensed]{svjour3}     
\smartqed  
\usepackage{graphicx}
\newtheorem{algorithm}{Algorithm}
\renewcommand{\thesection}{\arabic{section}.}

\renewcommand{\theenumi}{\roman{enumi}}
\begin{document} 

\title{Parallel projection methods for variational inequalities involving common fixed point problems
}

\titlerunning{Parallel methods for CSVIPs and CFPPs}        

\author{Dang Van Hieu $^*$ 
}

\authorrunning{D.V. Hieu} 

\institute{$^{*}$ Department of Mathematics, 
Vietnam National University, Hanoi \at
              334 Nguyen Trai, Thanh Xuan, Hanoi,
Vietnam \\
              Tel.: +84-979817776\\
              \email{dv.hieu83@gmail.com}
}

\date{Received: date / Accepted: date}

\maketitle

\begin{abstract} 
In this paper, we introduce two novel parallel projection methods for finding a solution of a system of 
variational inequalities which is also a common fixed point of a family of (asymptotically) $\kappa$ - strict pseudocontractive mappings. 
A technical extension in the proposed algorithms helps in computing practical numerical 
experiments when the number of subproblems is large. Some numerical examples are implemented to 
demonstrate the efficiency of parallel computations.
\end{abstract}
\section{Introduction}\label{intro}
Numerous problems in science and engineering, including optimization problems, fixed point problems, transportation problems, 
financial equilibrium problems, migration equilibrium problems etc. \cite{BO1994,DM1992,HS1966} lead to study variational 
inequality problems (VIP). Most of existing algorithms for solving VIPs in Hilbert space were based on the metric projection onto closed convex sets 
\cite{CGR2011a,HS1966,K1976,Y2001}.

In 1976, Korpelevich \cite{K1976} proposed the extragradient method for solving VIP for a Lipschitz continuous and monotone mapping $A$ 
on a closed convex set $C$ in Euclidean space and it was extended to Hilbert space by Nadezhkina and Takahashi \cite{NT2006a}. 
The projection plays an important role in constrained optimization problems. However, it is only found exactly when $C$ has a simple structure, 
for instance, as balls, hyperplanes or half-spaces. In 2011, authors 
in \cite{CGR2011a} proposed the subextragradient method where they replaced the second projection in the extragradient method onto $C$ by 
one onto a specially constructed halfspace. Moreover, they also introduced a modification of the subextragradient method \cite{CGR2011a} 
for finding a common point of the solution set of a variational inequality and the fixed point set of a nonexpansive mapping.

In recent years, the problem which is called the common solutions to variational inequality problems (CSVIP) \cite{CGRS2012} 
has been widely studied both theoretically and algorithmically. CSVIP is very general, in the sense that, it includes many special mathmatical models as: 
the convex feasibility problem, the common fixed point problem, the common minimizer problem, the common saddle point problem (CFPP), 
the variational inequality problem over the intersection of convex sets, the hierarchical variational inequality problem. Some algorithms 
proposed for solving CSVIP can be found in \cite{ABH2014,CGR2012b,CGRS2012,H2015,H2015a,Y2001}. Most of existing 
methods for CSVIP is inherently sequential. This will be costly on a single processor when the number of the subproblems of CSVIP 
is large. 

This paper focuses on the problem of finding a solution of CSVIP for Lipschitz continuous and monotone operators $A_i, i=1,\ldots,N$ 
involving (asymptotically) $\kappa$ - strict pseudocontractive mappings $S_j,j=1,\ldots,M$. Two parallel projection algorithms 
are proposed without using the product space \cite[Section 7.2]{CGR2012b} and their strong convergence is established. Our  
algorithms can be considered as the improvements of \cite[Algorithm 3.1]{CGRS2012} when we have replaced the extragradient method by 
the subextragradient method \cite[Algorithm 4.1]{CGR2011a} in which the second projection may be easily performed more on a specially 
constructed half-space. Moreover, using the parallel splitting-up technique in \cite[Algorithm 3.4]{H2015} we have designed the simultaneous 
iteration methods. Thus, numerical experiments can be implemented on computing clusters. A technical extension (see, Step 3 of Algorithm 
$\ref{Algor.1}$ below) can help us in numerical computations without solving optimization problems onto the intersection of $N+1$ sets 
when the number of the subproblems $N$ is large which is an obstacle in \cite[the final step of Algorithm 3.1]{CGRS2012}. In addition 
a minor level of generality from nonexpansive mappings to (asymptotically) $\kappa$ - strict pseudocontractive mappings is also studied in 
the proposed algorithms.

This paper is organized as follows: In Section $\ref{pre}$ we collect some definitions and priminary results used in the paper. 
Section $\ref{main}$ deals with proposing two parallel algorithms and analysising their convergence. In Section $\ref{numerical.example}$, 
several numerical experiments are illustrated for the efficiency of the proposed parallel hybrid algorithms.

\section{Preliminaries}\label{pre}
In this section, we recall some definitions and results for furtther researches. Let $H$ be a real Hilbert space with the inner product 
$\left\langle .,.\right\rangle$ and the induced norm $||.||$. We begin with some concepts of the monotonicity of an operator.
\begin{definition}\cite{R1970}
An operator $A:H \to H$ is said to be
\begin{itemize}
\item [$(i)$] monotone if $\left\langle A(x)-A(y),x-y\right\rangle\ge 0$, for all $x,y\in H$;
\item [$(ii)$] $\alpha$ - inverse strongly monotone if there exists a positive constant $\alpha$ such that
$$
\left\langle A(x)-A(y), x-y\right\rangle\ge \alpha||A(x)-A(y)||^2, \quad \forall x,y\in H;
$$
\item [$(iii)$] maximal monotone if it is monotone and its graph is not a proper subset of the one of any other monotone operator;
\item [$(iv)$] $L$ - Lipschitz continuous if there exists a positive constant $L$ such that $||A(x)-A(y)||\le L||x-y||$ for all $x,y\in H$.
\end{itemize}
\end{definition}
Let $C$ be a nonempty closed and convex subset of $H$. The variational inequality problem (VIP) for an operator $A$ on $C$ is to find $x^*\in C$ such that
\begin{equation}\label{eq:VIP}
\left\langle A(x^*),x-x^* \right\rangle \ge 0, \quad \forall x\in C.
\end{equation}
The set of solutions of VIP $(\ref{eq:VIP})$ is denoted by $VI(A,C)$. 
We have the following result concerning with the convexity and closedness of the solution set $VI(A,C)$.
\begin{lemma}\label{lem:T2000}\cite{T2000}
 Let $C$ be a nonempty, closed convex subset of a Hilbert space $H$ and $A$ be a monotone, hemicontinuous mapping from $C$ into $H$. Then
$$
VI(A,C)=\left\{u\in C:\left\langle v-u,A(v)\right\rangle\ge 0, \quad \forall v\in C\right\}.
$$
\end{lemma}
\begin{definition}\cite{BP1967,GK1972}
A mapping $S:H\to H$ is said to be
\begin{itemize}
\item [$(i)$] nonexpansive if  $||Sx-Sy||\le ||x-y||$ for all $x,y\in H;$
\item [$(ii)$] uniformly $L$ - Lipschitz continuous if there exists a constant $L>0$ such that
$$ ||S^nx-S^ny||\le L ||x-y||; $$
\item [$(iii)$] asymptotically nonexpansive if there exists a sequence $\left\{k_n\right\}\subset [1;+\infty)$ with $k_n\to 1$ such that 
$$ ||S^n x-S^n y||\le k_n ||x-y|| ,\quad \forall x,y\in H, n\ge 1;$$
\item [$(iv)$] $\kappa$-strict pseudocontractive if there exists a constant $\kappa\in [0;1)$ such that
$$ ||Sx-Sy||^2\le||x-y||^2+\kappa||(I-S)x-(I-S)y||^2,\quad \forall x,y\in H; $$
\item [$(v)$] asymptotically $\kappa$-strict pseudocontractive if there exist a constant $\kappa\in [0;1)$ and $\left\{k_n\right\}\subset [1;+\infty)$ with $k_n\to 1$ such that 
$$ ||S^nx-S^ny||^2\le k_n||x-y||^2+\kappa||(I-S^n)x-(I-S^n)y||^2,\forall x,y\in H, n\ge 1. $$
\end{itemize}
\end{definition}
\begin{lemma}\cite{SXY2009}\label{lem.demiclose}
Let H be a real Hilbert space, C be a nonempty closed convex subset of H and $S:C\to C$ be an asymptotically $\kappa$-strict 
pseudocontraction with the sequence $\left\{k_n\right\}\subset [1;\infty), k_n\to 1$ and the fixed point set $F(S)$. Then
\begin{enumerate}
\item [$(i)$] $F(S)$ is a closed convex subset of $H$.
\item [$(ii)$] $I-S$ is demiclosed, i.e., whenever $\left\{x_n\right\}$ is a sequence in $C$ weakly converging to some $x\in C$ and the sequence 
$\left\{(I-S)x_n\right\}$ strongly converges to some $y$, it follows that $(I-S)x=y$. 
\item [$(iii)$] S is uniformly $L$-Lipschitz continuous with the constant $$L=\sup\left\{\frac{\kappa+\sqrt{1+(1-\kappa)(k_n-1)}}{1+\kappa}:n\ge 1\right\}.$$
\end{enumerate}
\end{lemma}
It is easy to show that in any real Hilbert space, the following inequality holds
\begin{equation}\label{lem.aux}
||ax+(1-a)y||^2\le a||x||^2+(1-a)||y||^2-a(1-a)||x-y||^2,~\forall x,y\in H, a\in[0,1].
\end{equation}
For every $x\in H$, the metric projection $P_Cx$ of $x$ onto $C$ defined by 
$$
P_C x=\arg\min\left\{\left\|y-x\right\|:y\in C\right\}.
$$
Since C is a nonempty closed and convex subset of $H$, $P_C x$ exists and is unique. The projection $P_C:H\to C$ has the following characterizations:
\begin{lemma}\label{PropertyPC}
\begin{itemize}
\item [$(i)$] For all $y\in H, x\in C$, $\left\|x-P_C y\right\|^2+\left\|P_C y-y\right\|^2\le \left\|x-y\right\|^2.$
\item [$(ii)$] $z=P_C x$ if and only if $\left\langle x-z,z-y \right\rangle \ge 0,\quad \forall y\in C.$
\end{itemize}
\end{lemma}
The normal cone $N_C$ to a set $C$ at a point $ x \in C$ defined by
$$
N_C(x)=\left\{x^*\in H:\left\langle x-y,x^*\right\rangle\ge 0,\quad \forall y\in C\right\}.
$$
We have the following result.
\begin{lemma}\cite{R1970}\label{lem:N-Normal-Cone}
Let $C$ be a nonempty closed convex subset of a Hilbert space $H$ and let $A$ be a monotone and hemi-continuous mapping of $C$ into 
$H$ with $D(A) = C$. 
Let $Q$ be a mapping defined by:
$$
Q(x)=
\left \{
\begin{array}{ll}
Ax+N_C(x)\quad&\mbox{ if }\quad x\in C,\\
\emptyset  \quad& \mbox{if}\quad x\notin C.
\end{array}
\right.
$$
Then $Q$ is a maximal monotone and $Q^{-1}0=VI(A,C)$.
\end{lemma}
\section{Main results}\label{main}
In this section, we consider CSVIP \cite{CGRS2012} of finding $x^*\in K:=\cap_{i=1}^N K_i$ such that
\begin{equation}\label{eq:CSVIP}
\left\langle A_i(x^*), x-x^*\right\rangle\ge 0, \quad \forall x\in K_i, i=1,\ldots,N,
\end{equation}
where $K_i,~ i=1, \ldots,N$ are $N$ nonempty, closed and convex subsets of $H$ such that $K:=\cap_{i=1}^N K_i\ne \O$ and 
$A_i: H \to H$ is a Lipschitz continuous and monotone mapping for each $i=1, \ldots,N$. 
We present here two parallel projection algorithms for finding a solution of CSVIP $(\ref{eq:CSVIP})$ which is also a common fixed point 
of a finite family of (asymptotically) $\kappa$ - strict pseudocontractive mappings $\left\{S_j\right\}_{j=1}^M$. 
In the sequel, without loss of generality, we assume that the operators $\left\{A_i\right\}_{i=1}^N$ are Lipschitz continuous 
with a same constant $L>0$ and the mappings $\left\{S_j\right\}_{j=1}^M$ are asymptotically strict pseudocontractive with a same constant 
$\kappa\in [0,1)$ and a same sequence $\left\{k_n\right\}\subset [1,\infty)$ and $k_n\to 1$ as $n\to\infty$. Moreover, in Algorithm $\ref{Algor.1}$ 
below, the assumption of the boundedness of the solution set
$$
F=\left(\cap_{i=1}^N VI(A_i,K_i)\right)\bigcap \left(\cap_{j=1}^MF(S_j)\right)
$$
is required, i.e., there exists a positive real number $\omega$ such that $F\subset \Omega:=\left\{u\in H: ||u||\le \omega\right\}$.\\
\begin{algorithm}\label{Algor.1}
\textbf{Initialization:} $x_0\in H, n:=0$. The control parameter sequences $\left\{\alpha_k\right\},\left\{\beta_k\right\}$ satisfy the following conditions:
\begin{itemize}
\item [$(a)$] $0\le\alpha_k<1$, $\lim_{k\to\infty}\sup\alpha_k<1$;
\item [$(b)$] $\kappa\le\beta_k\le b<1$ for some $b\in (\kappa;1)$;
\item [$(c)$] $0<\lambda<\frac{1}{L}$.
\end{itemize}
\textbf{Step 1.} Find $N$ the projections $y_n^i$ on $K_i$ in parallel 
$$
y_n^i=P_{K_i}(x_n-\lambda A_i(x_n)), i=1,\ldots,N.
$$
\textbf{Step 2.} Find $N$ the projections $z_n^i$ on the half-space $T_n^i$ in parallel
$$
z_n^i=P_{T_n^i}(x_n-\lambda A_i(y_n^i)), i=1,\ldots,N,
$$
where $T_n^i=\left\{v\in H: \left\langle(x_n-\lambda A_i(x_n))-y_n^i,v-y_n^i\right\rangle\le 0\right\}$.\\
\textbf{Step 3.} Find the furthest element from $x_n$ among all $z_n^i$, i.e., 
$$ i_n = {\rm argmax}\{||z_n^i - x_n||: i =1,\ldots,N\},\bar{z}_n:=z^{i_n}_n. $$
\textbf{Step 4.} Find intermediate approximations $u_n^j$ in parallel
$$ u_n^j=\alpha_n x_n+(1-\alpha_n)\left(\beta_n\bar{z}_n+(1-\beta_n) S_j^n \bar{z}_n\right), j=1,\ldots,M. $$
\textbf{Step 5.} Find the furthest element from $x_n$ among all $u_n^j$, i.e., 
$$ j_n= {\rm argmax}\{||u_n^j - x_n||: j =1,\ldots,M\},\bar{u}_n:=u^{j_n}_n. $$
\textbf{Step 6.} Compute $x_{n+1}= P_{C_{n}\cap Q_n}(x_0)$ where $
C_{n} = \{v \in H: ||\bar{u}_n - v||^2\leq ||x_n-v||^2+\epsilon_n\},
Q_{n}= \{v \in H: \left\langle v-x_n,x_n-x_0 \right\rangle\ge 0\}$ and
 $\epsilon_n=(k_n-1)\left(||x_n||+\omega\right)^2$. Set $n:=n+1$ and go back \textbf{Step 1}.
\end{algorithm}
In order to prove the convergence of Algorithm $\ref{Algor.1}$, we need the following lemmas.
\begin{lemma}\label{lem.1} Suppose that $x^* \in F$ and the sequences $\left\{y_n^i\right\}, \left\{z_n^i\right\}$ generated by Step 1 and Step 2 of Algorithm $\ref{Algor.1}$. Then
$$
\left\|z_n^i-x^*\right\|^2\le \left\|x_n-x^*\right\|^2-c\left(\left\|y_n^i-x_n\right\|^2+\left\|z_n^i-y_n^i\right\|^2\right),
$$
where $c=1-\lambda L>0$.
\end{lemma}
\begin{proof}
Since $A_i$ is monotone on $K_i$ and $y_n^i\in K_i$, we obtain
$$
\left\langle A_i(y_n^i)-A_i(x^*), y_n^i-x^*\right\rangle\ge 0, \quad \forall x^*\in F.
$$
This together with $x^*\in VI(A,K_i)$ implies that $\left\langle A_i(y_n^i), y_n^i-x^*\right\rangle\ge 0.$ Thus, 
\begin{equation}\label{eq:2*}
\left\langle A_i(y_n^i), z_n^i-x^*\right\rangle\ge \left\langle A_i(y_n^i), z_n^i-y_n^i\right\rangle .
\end{equation}
From the characterization of the metric projection onto $T_n^i$, we have
$$
\left\langle z_n^i-y_n^i, (x_n-\lambda A_i(x_n))-y_n^i\right\rangle\le 0.
$$
Thus
\begin{equation}
\begin{array}{ll}
\left\langle z_n^i-y_n^i, (x_n-\lambda A_i(y_n^i))-y_n^i\right\rangle&=\left\langle z_n^i-y_n^i, (x_n-\lambda A_i(x_n))-y_n^i\right\rangle\nonumber\\
&+\lambda\left\langle z_n^i-y_n^i, A_i(x_n)-A_i(y_n^i)\right\rangle\nonumber\\
&\le\lambda\left\langle z_n^i-y_n^i, A_i(x_n)-A_i(y_n^i)\right\rangle.\label{eq:4}
\end{array}
\end{equation}
Putting $t_n^i=x_n-\lambda A_i(y_n^i)$ and rewriting $z_n^i=P_{T_n^i}(t_n^i)$. From Lemma $\ref{PropertyPC}$ and $(\ref{eq:2*})$, we get
\begin{equation}
\begin{array}{ll}
||z_n^i-x^*||^2&\le ||t_n^i-x^*||^2-||P_{T_n^i}(t_n^i)-t_n^i||^2\nonumber\\
&=||x_n-\lambda A_i(y_n^i)-x^*||^2-||z_n^i-(x_n-\lambda A_i(y_n^i))||^2\nonumber\\
&=||x_n-x^*||^2-||z_n^i-x_n||^2+2\lambda\left\langle x^*-z_n^i, A_i(y_n^i)\right\rangle\nonumber\\
&\le||x_n-x^*||^2-||z_n^i-x_n||^2+2\lambda\left\langle y_n^i-z_n^i, A_i(y_n^i)\right\rangle.\label{eq:5}
\end{array}
\end{equation}
From $(\ref{eq:4})$, we also have
\begin{eqnarray}
||z_n^i-x_n||^2&-&2\lambda\left\langle y_n^i-z_n^i, A_i(y_n^i)\right\rangle\nonumber\\ 
&=&||z_n^i-y_n^i+y_n^i-x_n||^2-2\lambda\left\langle y_n^i-z_n^i, A_i(y_n^i)\right\rangle\nonumber\\ 
&=&||z_n^i-y_n^i||^2+||y_n^i-x_n||^2-2\left\langle z_n^i-y_n^i,\left(x_n-\lambda A_i(y_n^i)-y_n^i\right)\right\rangle\nonumber\\ 
&=&||z_n^i-y_n^i||^2+||y_n^i-x_n||^2-2\lambda\left\langle z_n^i-y_n^i, A_i(x_n)-A_i(y_n^i)\right\rangle\nonumber\\
&\ge&||z_n^i-y_n^i||^2+||y_n^i-x_n||^2-2\lambda ||z_n^i-y_n^i||||A_i(x_n)-A_i(y_n^i)||\nonumber\\
&\ge&||z_n^i-y_n^i||^2+||y_n^i-x_n||^2-2L\lambda ||z_n^i-y_n^i||||x_n-y_n^i||\nonumber\\
&\ge&||z_n^i-y_n^i||^2+||y_n^i-x_n||^2-L\lambda \left(||z_n^i-y_n^i||^2+||x_n-y_n^i||^2\right)\nonumber\\
&\ge&(1-L\lambda)||z_n^i-y_n^i||^2+(1-L\lambda)||y_n^i-x_n||^2 \nonumber\\
&=&c\left(||z_n^i-y_n^i||^2+||y_n^i-x_n||^2\right). \nonumber
\end{eqnarray}
This together with $(\ref{eq:5})$ ensures the truth of Lemma $\ref{lem.1}$.
\end{proof}
\begin{lemma}\label{lem.2}
Suppose that Algorithm $\ref{Algor.1}$ reaches to the iteration $n$. Then $F\subset C_n\cap Q_n$ and $x_{n+1}$ is well-defined.
\end{lemma}
\begin{proof}
Since $A_i$ is Lipschitz continuous, $A_i$ is continuous. Lemma $\ref{lem:T2000}$ ensures that $VI(A_i, K_i)$ is closed and convex for 
all $i=1,\ldots,N$. From Lemma $\ref{lem.demiclose}$, we see that $F(S_j)$ is also closed and convex for all $j=1,\ldots,M$. Thus, $F$ is 
closed and convex. From the definitions of $C_n$ and $Q_n$, we see that $Q_n$ is closed and convex and $C_n$ is closed. On the other hand, 
the relation $||\bar{u}_n - v||^2\leq ||x_n-v||^2+\epsilon_n$ is equivalent to 
$$ 2\left\langle v,x_n-\bar{u}_n\right\rangle\le ||x_n||^2-||\bar{u}_n||^2+\epsilon_n. $$
This implies that $C_n$ is convex. Putting $S_{j,\beta_n}=\beta_n I+(1-\beta_n) S_j^n $ and rewriting 
$u_n^j=\alpha_n x_n+(1-\alpha_n)S_{j,\beta_n}\bar{z}_n$. From the convexity of $||.||^2$ and the relation $(\ref{lem.aux})$, for each $u\in F$ we have
\begin{eqnarray}
||\bar{u}_n-u||^2&=&||\alpha_n x_n+(1-\alpha_n)S_{j_n,\beta_n}\bar{z}_n-u||^2\nonumber\\
&\le&\alpha_n||x_n-u||^2+(1-\alpha_n)||S_{j_n,\beta_n}\bar{z}_n-u||^2\nonumber\\
&=&\alpha_n||x_n-u||^2+(1-\alpha_n)||\beta_n \bar{z}_n+(1-\beta_n) S_{j_n}^n\bar{z}_n-u||^2\nonumber\\
&=&\alpha_n||x_n-u||^2\nonumber\\
&&+(1-\alpha_n)\left(\beta_n ||\bar{z}_n-u||^2+(1-\beta_n)||S_{j_n}^n\bar{z}_n-S_{j_n}^nu||^2\right)\nonumber\\
&&-(1-\alpha_n)\beta_n(1-\beta_n)||(\bar{z}_n-u)-(S_{j_n}^n\bar{z}_n-S_{j_n}^nu)||^2\nonumber\\
&=&\alpha_n||x_n-u||^2+(1-\alpha_n)\left(\beta_n ||\bar{z}_n-u||^2+(1-\beta_n)k_n||\bar{z}_n-u||^2\right)\nonumber\\
&&+\kappa(1-\alpha_n)(1-\beta_n)||(I-S_{j_n}^n)\bar{z}_n-(I-S_{j_n}^n)u||^2\nonumber\\
&&-(1-\alpha_n)\beta_n(1-\beta_n)||(\bar{z}_n-u)-(S_{j_n}^n\bar{z}_n-S_{j_n}^nu)||^2\nonumber\\
&=&\alpha_n||x_n-u||^2+(1-\alpha_n)\left(\beta_n ||\bar{z}_n-u||^2+(1-\beta_n)k_n||\bar{z}_n-u||^2\right)\nonumber\\
&&-(\beta_n-\kappa)(1-\alpha_n)(1-\beta_n)||(I-S_{j_n}^n)\bar{z}_n-(I-S_{j_n}^n)u||^2\nonumber\\
&\le&\alpha_n||x_n-u||^2+(1-\alpha_n)\left(\beta_n ||\bar{z}_n-u||^2+(1-\beta_n)k_n||\bar{z}_n-u||^2\right)\nonumber\\
&=&\alpha_n||x_n-u||^2+(1-\alpha_n)||\bar{z}_n-u||^2+(1-\beta_n)(k_n-1)||\bar{z}_n-u||^2\nonumber\\
&\le&\alpha_n||x_n-u||^2+(1-\alpha_n)||\bar{z}_n-u||^2+(k_n-1)||\bar{z}_n-u||^2. \label{eq:6}
\end{eqnarray}
From Lemma $\ref{lem.1}$, we obtain $||\bar{z}_n-u||\le||x_n-u||$. This together with $(\ref{eq:6})$ implies that
$$
\begin{array}{ll}
||\bar{u}_n-u||^2&\le\alpha_n||x_n-u||^2+(1-\alpha_n)||x_n-u||^2+(k_n-1)||x_n-u||^2\nonumber\\
&\le||x_n-u||^2+(k_n-1)\left(||x_n||+||u||\right)^2\nonumber\\
&=||x_n-u||^2+\epsilon_n.\nonumber
\end{array}
$$
Thus, $F\subset C_n$ for all $n\ge 0$. Now, we show $F\subset C_n\cap Q_n$ for all $n\ge 0$ by the induction. Indeed, $F\subset Q_0$ and so $F\subset C_0\cap Q_0$. Assume that $F\subset C_{n}\cap Q_{n}$ for some $n\ge 0$. From $x_{n+1}=P_{C_{n}\cap Q_n}x_0$ and Lemma $\ref{PropertyPC}$, we obtain
$$ \left\langle v-x_{n+1}, x_{n+1}-x_0\right\rangle\ge 0, \quad \forall v\in C_{n}\cap Q_{n}. $$
Since $F\subset C_n\cap Q_n$, $\left\langle v-x_{n+1}, x_{n+1}-x_0\right\rangle\ge 0$ for all $ v \in F$. This together with the definition of 
$Q_{n+1}$ implies that  $F\subset Q_{n+1}$. Thus, by the induction $F\subset C_n\cap Q_n$ for all $n\ge 0$. Since $F\ne \O$, $P_F x_0$ 
and $x_{n+1}=P_{C_n\cap Q_n}x_0$ are well-defined.
\end{proof}
\begin{lemma}\label{lem.3}
Suppose that $\left\{x_n\right\},\left\{y_n^i\right\}, \left\{z_n^i\right\}$ and $\left\{u_n^j\right\}$ generated by  Algorithm $\ref{Algor.1}$. Then there hold the relations
$$
\lim_{n\to \infty}||x_{n+1}-x_n||=\lim_{n\to \infty}||y_n^i-x_n||=\lim_{n\to \infty}||z_{n}^i-x_n||=\lim_{n\to \infty}||u_{n}^j-x_n||=0,
$$
and $\lim_{n\to \infty}||S_j x_{n}-x_n||=0$ for all $i=1,\ldots,N, j=1,\ldots,M$.
\end{lemma}
\begin{proof}
We have $x_n=P_{Q_n}x_0$ and $F\subset Q_n$. For each $u\in F$, from Lemma $\ref{PropertyPC}$, we have
\begin{equation}\label{eq:8}
||x_n-x_0||\le ||u-x_0||, \quad \forall n\ge 0.
\end{equation}
Thus, the sequence $\left\{||x_n-x_0||\right\}$, and so $\left\{x_n\right\}$, are bounded. From $x_{n+1}\in Q_n$ and $x_n=P_{Q_n}x_0$, we also obtain
$
||x_n-x_0||\le ||x_{n+1}-x_0|| 
$
for all $n\ge 0$. This implies that the sequence $\left\{||x_n-x_0||\right\}$ is nondecreasing. Therefore, there exists the limit of the sequence $\left\{||x_n-x_0||\right\}$. Moreover, from $x_{n+1}\in Q_n$ and $x_n=P_{Q_n}x_0$, we get
$$
||x_n-x_{n+1}||^2\le ||x_{n+1}-x_0||^2-||x_n-x_0||^2.
$$
From this inequality, letting $n\to\infty$, we find
\begin{equation}\label{eq:11}
\lim_{n\to\infty}||x_n-x_{n+1}||=0.
\end{equation}
From the definition of $C_{n}$ and $x_{n+1}\in C_n$, we have
\begin{equation}\label{eq:12}
||\bar{u}_n-x_{n+1}||^2\le ||x_n-x_{n+1}||^2+\epsilon_n.
\end{equation}
From the boundedness of $\left\{x_n\right\}$, we find that $\epsilon_n=(k_n-1)(||x_n||+\omega)^2\to 0$ as $n\to \infty$. This together with the relations $(\ref{eq:11}),(\ref{eq:12})$ implies that
$
||\bar{u}_n-x_{n+1}||\to 0.
$
Thus,
\begin{equation}\label{eq:14}
\lim_{n\to\infty}||\bar{u}_n-x_{n}||=0.
\end{equation}
From the definition of $j_n$, we get
$
\lim_{n\to\infty}||u_n^j-x_{n}||=0, j=1,\ldots, M.
$
From $(\ref{eq:6})$ and Lemma $\ref{lem.1}$, we have
\begin{eqnarray*}
||\bar{u}_n-u||^2&\le&\alpha_n||x_n-u||^2+(1-\alpha_n)||\bar{z}_n-u||^2+(1-\beta_n)(k_n-1)||\bar{z}_n-u||^2\\
&\le&\alpha_n||x_n-u||^2+(1-\alpha_n)||\bar{z}_n-u||^2+(k_n-1)||x_n-u||^2\\
&\le&||x_n-u||^2-c(1-\alpha_n)\left(\left\|y_n^{i_n}-x_n\right\|^2+\left\|\bar{z}_n-y_n^{i_n}\right\|^2\right)+\epsilon_n.
\end{eqnarray*}
Hence,
\begin{equation}\label{eq:16}
c(1-\alpha_n)\left(\left\|y_n^{i_n}-x_n\right\|^2+\left\|\bar{z}_n-y_n^{i_n}\right\|^2\right)\le||x_n-u||^2-||\bar{u}_n-u||^2+\epsilon_n.
\end{equation}
Moreover, 
\begin{eqnarray*}
\left|||x_n-u||^2-||\bar{u}_n-u||^2\right|&=&\left|||x_n-u||-||\bar{u}_n-u||\right|\left(||x_n-u||+||\bar{u}_n-u||\right)\\
&\le& ||x_n-\bar{u}_n||\left(||x_n-u||+||\bar{u}_n-u||\right).
\end{eqnarray*}
The last inequality together with $(\ref{eq:14})$ and the boundedness of the sequences $\left\{x_n\right\}, \left\{\bar{u}_n\right\}$ implies that
\begin{equation}\label{eq:17}
\lim_{n\to\infty}\left(||x_n-u||^2-||\bar{u}_n-u||^2\right)=0.
\end{equation}
From $(\ref{eq:16}),(\ref{eq:17})$ and $\lim\sup_{n\to\infty}\alpha_n<1$, $\epsilon_n\to 0$ we obtain
\begin{equation}\label{eq:18}
\lim_{n\to\infty}\left\|y_n^{i_n}-x_n\right\|=\lim_{n\to\infty}\left\|\bar{z}_n-y_n^{i_n}\right\|=0.
\end{equation}
Thus,
\begin{equation}\label{eq:19}
\lim_{n\to\infty}\left\|\bar{z}_n-x_n\right\|=0
\end{equation}
because of $\left\|\bar{z}_n-x_n\right\|\le\left\|\bar{z}_n-y_n^{i_n}\right\|+\left\|y_n^{i_n}-x_n\right\|$. 
From the definition of $i_n$, we obtain
\begin{equation}\label{eq:20}
\lim_{n\to\infty}\left\|z^i_n-x_n\right\|=0, i=1,\ldots,N.
\end{equation}
From Lemma $\ref{lem.1}$ and $(\ref{eq:20})$, by arguing similarly as $(\ref{eq:18})$ one has
$$
\lim_{n\to\infty}\left\|y^i_n-x_n\right\|=0, i=1,\ldots,N.
$$
From $u_n^j=\alpha_n x_n+(1-\alpha_n)S_{j,\beta_n}\bar{z}_n$ we get $||u_n^j-x_n||=(1-\alpha_n)||S_{j,\beta_n}\bar{z}_n-x_n||$. 
By $\lim\sup_{n\to\infty}\alpha_n<1$ and $||u_n^j-x_n||\to 0$ we obtain
\begin{equation}\label{eq:22}
\lim_{n\to\infty}||S_{j,\beta_n}\bar{z}_n-x_n||=0.
\end{equation}
The relations $(\ref{eq:19})$ and $(\ref{eq:22})$ lead to
$
\lim_{n\to\infty}||S_{j,\beta_n}\bar{z}_n-\bar{z}_n||=0.
$
 Thus, from the definition of $S_{j,\beta_n}$, one has
$
\lim_{n\to\infty}(1-\beta_n)||S_j^n\bar{z}_n-\bar{z}_n||=0.
$
Since $\beta_n\le b<1$,
\begin{equation}\label{eq:24}
\lim_{n\to\infty}||S_j^n\bar{z}_n-\bar{z}_n||=0.
\end{equation}
From Lemma $\ref{lem.demiclose}$, $S_j$ is uniformly L - Lipschitz continuous. Therefore
\begin{eqnarray*}
||S_j^nx_n-x_n||&\le& ||S_j^nx_n-S_j^n\bar{z}_n||+||S_j^n\bar{z}_n-\bar{z}_n||+||\bar{z}_n-x_n||\\ 
&\le& L||x_n-\bar{z}_n||+||S_j^n\bar{z}_n-\bar{z}_n||+||\bar{z}_n-x_n||\\ 
&\le& (L+1)||x_n-\bar{z}_n||+||S_j^n\bar{z}_n-\bar{z}_n||.
\end{eqnarray*}
The last inequality together with $(\ref{eq:19}), (\ref{eq:24})$, we have
\begin{equation}\label{eq:25}
\lim_{n\to\infty}||S_j^nx_n-x_n||=0, j=1,\ldots,M.
\end{equation}
On the other hand,
\begin{eqnarray*}
||S_{j}x_n-x_n||&\le&||S_{j}x_n-S_j^{n+1}x_n||+||S_{j}^{n+1}x_n-S_j^{n+1}x_{n+1}||\\
&&+||S_{j}^{n+1}x_{n+1}-x_{n+1}||+||x_{n+1}-x_n||\\ 
&\le&L||x_n-S_j^n x_n|| +(L+1)||x_{n}-x_{n+1}||+||S_{j}^{n+1}x_{n+1}-x_{n+1}||.
\end{eqnarray*}
This together with $(\ref{eq:12})$ and $(\ref{eq:25})$ implies that
\begin{equation}\label{eq:26}
\lim_{n\to\infty}||S_{j}x_n-x_n||=0.
\end{equation}
The proof of Lemma $\ref{lem.3}$ is complete.
\end{proof}
\begin{lemma}\label{lem.4}
Suppose that $p$ is a weak cluster point of the sequence $\left\{x_n\right\}$. Then $p\in F$.
\end{lemma}
\begin{proof}
Since $\left\{x_n\right\}$ is bounded, there exists a subsequence of $\left\{x_n\right\}$ weakly converging to $p$. For the sake of simplicity, we denote this subsequence again by $\left\{x_n\right\}$. From $(\ref{eq:26})$ and the demicloseness of $S_j$, we obtain $p\in F(S_j)$. So, $p\in \cap_{j=1}^M F(S_j)$.

Now we prove that $p\in \cap_{i=1}^N VI(A_i,K_i)$. Indeed, for each 
operator $A_i$, we define the mapping $Q_i$ by
$$
Q_i(x)=\left \{
\begin{array}{ll}
&A_ix+N_{K_i}(x)\quad\mbox{ if }\quad x\in K_i\\
&\O \quad\quad\quad\quad\quad\quad \mbox{if}\quad x\notin K_i,
\end{array}
\right.
$$
where $N_{K_i}(x)$ is the normal cone to  $K_i$ at $x\in K_i$. Lemma $\ref{lem:N-Normal-Cone}$ ensures that $Q_i$ is maximal monotone. 
For each pair $(x,y)$ in the graph of $Q_i$, i.e., $(x,y)\in G(Q_i)$, from the definition of the mapping $Q_i$ we see that 
$y-A_i(x)\in N_{K_i}(x)$. Therefore,
$$\left\langle x-z,y-A_i(x)\right\rangle\ge 0$$
for all $z\in K_i$ because of the definition of $N_{K_i}(x)$. Substituting $z=y_n^i\in K_i$ onto the last inequality, one gets
$
\left\langle x-y_n^i,y-A_i(x)\right\rangle\ge 0.
$
Therefore,
\begin{equation}\label{eq:24}
\left\langle x-y_n^i,y\right\rangle\ge\left\langle x-y_n^i,A_i(x)\right\rangle.
\end{equation}
By $y_n^i=P_{K_i}\left(x_n-\lambda A_i x_n\right)$ and Lemma $\ref{PropertyPC}$(ii),
$\left\langle x-y_n^i,y_n^i-x_n+\lambda A_i x_n\right\rangle\ge 0.$
A straightforward computation yields
\begin{equation}\label{eq:25}
\left\langle x-y_n^i,A_i(x_n)\right\rangle \ge \left\langle x-y_n^i,\frac{x_n-y_n^i}{\lambda}\right\rangle.
\end{equation}
Thus, from $(\ref{eq:24}),(\ref{eq:25})$ and the monotonicity of $A_i$, we find that
\begin{eqnarray}
\left\langle x-y_n^i,y\right\rangle&\ge&\left\langle x-y_n^i,A_i(x)\right\rangle\nonumber\\ 
&=& \left\langle x-y_n^i,A_i(x)-A_i(y_n^i)\right\rangle+\left\langle x-y_n^i,A_i(y_n^i)-A_i(x_n)\right\rangle\nonumber\\
&&+\left\langle x-y_n^i,A_i(x_n)\right\rangle\nonumber\\
&\ge& \left\langle x-y_n^i,A_i(y_n^i)-A_i(x_n)\right\rangle+\left\langle x-y_n^i,\frac{x_n-y_n^i}{\lambda}\right\rangle. \label{eq:26}
\end{eqnarray}
From the Lipschitz-continuity of $A_i$ and $||x_n-y_n^i||\to 0$, 
\begin{equation}\label{eq:28}
\lim_{n\to\infty}||A_i(y_n^i)-A_i(x_n)||=0.
\end{equation}
Since $||x_n-y_n^i||\to 0$ and $x_n\rightharpoonup p$, $y_n^i\rightharpoonup p$. 
Passing the limit in inequality $(\ref{eq:26})$ as $n\to\infty$ and employing $(\ref{eq:28})$, 
we obtain $\left\langle x-p,y\right\rangle\ge 0$ for all $(x,y)\in G(Q_i)$. Thus, from the maximal monotonicity of $Q_i$ and 
Lemma $\ref{lem:N-Normal-Cone}$, one has 
$p\in Q_i^{-1}0=VI(A_i,K_i)$ for all $1\le i\le N$. Hence, $p\in F$. The proof of Lemma $\ref{lem.4}$ is complete.
\end{proof}
\begin{theorem}\label{theo.1}
Let $K_i,i=1,\ldots,N$ be closed and convex subsets of real Hilbert space $H$ such that $K=\cap_{i=1}^N K_i \ne \O$. 
Suppose that $\left\{A_i\right\}_{i=1}^N:H \to H$ is a finite family of monotone and $L$ - Lipschitz continuous mappings 
and $\left\{S_j\right\}_{i=1}^M:H \to H$ is a finite family of asymptotically $\kappa$ - strict pseudocontractive mappings. 
In addition, the set $F$ is nonempty and bounded. Then, the sequences $\left\{x_n\right\}, \left\{y_n^i\right\}, \left\{z^i_n\right\}$ 
and $\left\{u^j_n\right\}$ generated by Algorithm $\ref{Algor.1}$ converge strongly to $P_F x_0$.
\end{theorem}
\begin{proof}
By Lemma $\ref{lem.2}$, $F, C_n, Q_n$ are nonempty closed and convex subsets of $H$. Besides, $F\subset C_n\cap Q_n$ for all $n\ge 0$. 
Therefore, $P_F x_0, P_{C_n \cap Q_n} x_0$ are well-defined. From Lemma $\ref{lem.3}$, $\left\{x_n\right\}$ is bounded. Assume that $p$ 
is any weak cluster point of $\left\{x_n\right\}$ and $x_{n_j}\rightharpoonup p$. By Lemma $\ref{lem.4}$, $p\in F$. We now show that 
$x_n\to P_{F}x_0$. Indeed, setting $x^\dagger:=P_{F}x_0$, from $(\ref{eq:8})$ and $x^\dagger\in F$ we have
$
||x_n-x_0||\le||x^\dagger-x_0||
$
for all $n\ge 0$. This together with the lower weak semicontinuity of the norm implies that
$$
||p-x_0||\le\lim\inf_{j\to \infty}||x_{n_j}-x_0||\le \lim\sup_{j\to \infty}||x_{n_j}-x_0|| \le||x^\dagger-x_0||.
$$
So, by the definition of $x^\dagger$, $p=x^\dagger$ and $\lim_{j\to \infty}||x_{n_j}-x_0|| =||x^\dagger-x_0||$. Therefore, $\lim_{j\to \infty}||x_{n_j}|| =||x^\dagger||$. From the Kadec-Klee property of the Hilbert space $H$, $x_{n_j}\to x^\dagger$. Thus, $x_n\to x^\dagger$. Lemma $\ref{lem.4}$ ensures that the sequences $\left\{y_n^i\right\},\left\{z_n^i\right\}, \left\{u_n^j\right\}$ also converge strongly to $P_F x_0$. The proof of Theorem $\ref{theo.1}$ is complete.
\end{proof}
Next, we consider CSVIP for the monotone and Lipschitz continuous operators $\left\{A_i\right\}_{i=1}^N$ involving a finite family of 
$\kappa$ - strict pseudocontractive mappings $\left\{S_j\right\}_{i=1}^M$. In Algorithm $\ref{Algor.2}$ below, the assumption of the 
boundedness of the solution set $F$ is not required. We have the following algorithm whose idea is similar to the Algorithm $\ref{Algor.1}$.
\begin{algorithm}\label{Algor.2}
 \textbf{Initialization:} $x_0\in H, n:=0$. The control parameter sequences $\left\{\alpha_k\right\},\left\{\beta_k\right\}$ satisfy the following conditions
\begin{itemize}
\item [$(a)$] $0\le\alpha_k<1$, $\lim_{k\to\infty}\sup\alpha_k<1$;
\item [$(b)$] $\kappa\le\beta_k\le b<1$ for some $b\in (\kappa;1)$;
\item [$(c)$] $0<\lambda<1/L$.
\end{itemize}
\textbf{Step 1.} Find $N$ the projections $y_n^i$ on $K_i$ in parallel
\begin{eqnarray*}
y_n^i=P_{K_i}(x_n-\lambda A_i(x_n)), i=1,\ldots,N.
\end{eqnarray*}
\textbf{Step 2.} Find $N$ the projections $z_n^i$ on the half-space $T_n^i$ in parallel
\begin{eqnarray*}
z_n^i=P_{T_n^i}(x_n-\lambda A_i(y_n^i)), i=1,\ldots,N,
\end{eqnarray*}
where $T_n^i$ is defined as in Algorithm $\ref{Algor.1}$.\\
\textbf{Step 3.} Find the furthest element from $x_n$ among all $z_n^i$, i.e., 
$$ i_n = {\rm argmax}\{||z_n^i - x_n||: i =1,\ldots,N\},\bar{z}_n:=z^{i_n}_n. $$
\textbf{Step 4.} Find intermediate approximations $u_n^j$ in parallel
$$ u_n^j=\alpha_n x_n+(1-\alpha_n)\left(\beta_n\bar{z}_n+(1-\beta_n) S_j\bar{z}_n\right), j=1,\ldots,M. $$
\textbf{Step 5.} Find the furthest element from $x_n$ among all $u_n^j$, i.e., 
$$ j_n= {\rm argmax}\{||u_n^j - x_n||: j =1,\ldots,M\},\bar{u}_n:=u^{j_n}_n. $$
\textbf{Step 6.} Compute $ x_{n+1}= P_{C_{n}\cap Q_n}(x_0)$ where 
$C_{n} = \{v \in H: ||\bar{u}_n - v||\leq ||x_n-v||\},
Q_{n}= \{v \in H: \left\langle v-x_n,x_n-x_0 \right\rangle\ge 0\}$. Set $n:=n+1$ and go back \textbf{Step 1}.
\end{algorithm}
We also have the following result.
\begin{theorem}\label{theo.2}
Let $K_i,i=1,\ldots,N$ be closed and convex subsets of real Hilbert space $H$ such that $K=\cap_{i=1}^N K_i \ne \O$. Suppose that $\left\{A_i\right\}_{i=1}^N:H \to H$ is a finite family of monotone and $L$ - Lipschitz continuous mappings and $\left\{S_j\right\}_{i=1}^M:H \to H$ is a finite family of $\kappa$ - strict pseudocontractive mappings. In addition, the set $F$ is nonempty. Then, the sequences $\left\{x_n\right\}, \left\{y_n^i\right\}, \left\{z^i_n\right\}$ and $\left\{u^j_n\right\}$ generated by Algorithm $\ref{Algor.2}$ converge strongly to $P_F x_0$.
\end{theorem}
\begin{proof}
Since $S_j$ is $\kappa$ - strict pseudocontractive mapping, $S_j$ is asymptotically $\kappa$ - strict pseudocontractive mapping with $k_n=1$ 
for all $n\ge 0$. Putting $\epsilon_n=0$, by arguing similarly as in the proof of Theorem $\ref{theo.1}$ we come to the desired conclusion.
\end{proof}
Using Theorem $\ref{theo.2}$, we obtain the following corollary.
\begin{corollary}\label{cor1}
Let $H$ be a real Hilbert space and $K_i,i=1,\ldots,N$ be closed convex subsets of $H$ such that $K=\cap_{i=1}^N K_i \ne \O$. Suppose that $A_i:H \to H$ is a monotone and $L$ - Lipschitz continuous mapping for each $i=1,2,\ldots,N$. In addition, the solution set $F=\cap_{i=1}^N VI(A_i,K_i)$ is nonempty. Let $\left\{x_n\right\}, \left\{y_n^i\right\}, \left\{z^i_n\right\}$ be the sequences generated by the following manner:
$$
\left\{
\begin{array}{ll}
&x_0\in H,\\ 
&y_n^i=P_{K_i}(x_n-\lambda A_i(x_n)), i=1,2,\ldots,N,\\
&z_n^i=P_{T^i_n}(x_n-\lambda A_i(y_n^i)), i=1,2,\ldots,N,\\
&i_n=\arg\max\left\{||z_n^i-x_n||:i=1,2,\ldots,N\right\}, \bar{z}_n=z_n^{i_n},\\
&C_{n} = \{v \in H: ||\bar{z}_n - v||\leq ||x_n-v||\},\\
&Q_{n}= \{v \in H: \left\langle v-x_n,x_n-x_0 \right\rangle\ge 0\},\\
& x_{n+1}= P_{C_{n}\cap Q_n}(x_0),n\ge 0,
\end{array}
\right.
$$
where $0<\lambda<1/L$ and $T_n^i$ is defined as in Algorithm $\ref{Algor.1}$. Then, the sequences $\left\{x_n\right\}, \left\{y_n^i\right\}, \left\{z^i_n\right\}$ converge strongly to $P_Fx_0$.
\end{corollary}
\begin{remark}
Corollary $\ref{cor1}$ can be considered as an improvement of Algorithm 3.1 in \cite{CGRS2012} in the following aspects:
\begin{itemize}
\item The second projection of the extragradient method \cite{K1976} on any closed convex set $K_i$ is replaced by one on 
the half-space $T_n^i$ which is easily performed more.
\item Chosing the furthest element $\bar{z}_n$ from $x_n$ among all $z_n^i$ is a technical extension. This can help us in 
implementing numerical experiments when the number of subproblems $N$ is large without solving distance optimization problems 
on the intersection of $N+1$ closed convex sets as Algorithm 3.1 in \cite{CGRS2012}.
\item Two sets $C_n$ and $Q_n$ are either the half-spaces or the whole space $H$. Thus, by using the same techniques as in 
\cite{SS2000}, we can obtain an explicit formula of the next iterate $x_{n+1}$ which is the projection of $x_0$ on the intersection 
$C_n\cap Q_n$ (see, the numerical experiments in Section $\ref{numerical.example}$). In order to obtain the same result in 
Algorithm 3.1 \cite{CGRS2012}, the number of subcases generated by the distance optimization problem is $2^{N+1}$. This is 
a main obstacle of Algorithm 3.1 \cite{CGRS2012} in numerical experiments when $N$ is large.
\end{itemize}
\end{remark}

\section{Numerical experiments}\label{numerical.example}
In this section, we consider two numerical examples to illustrate the convergence and the ability of the implementation 
of the proposed algorithms when the numbers of subproblems are large. In these cases, Algorithm 3.1 \cite{CGRS2012} and Algorithm 4.4 for CSVIP 
(see, Section 7.2 in \cite{CGR2012b}) seem to be difficult 
to practice numerical computations.

\textit{Example 1.} We consider a simple example in $\Re^3$ for $A_i=0, S_j=I$ for all $i,j$ and $K_i$ are balls $K_i=\left\{x\in R^3:||x-a_i||\le r_i\right\}$ 
centered at $a_i$ and the radius $r_i$ for $i=1,\ldots,N$. By Theorem $\ref{theo.2}$, the sequence $\left\{x_n\right\}$ generated by Algorithm $\ref{Algor.2}$ 
converges strongly to $P_{K}(x_0)$. According to  Algorithm $\ref{Algor.2}$, we see that $y_n^i=z_n^i=x_n$ if $||x_n-a_i||\le r_i$. Otherwise, 
$y_n^i=z_n^i=\frac{r_i}{||x_n-a_i||}(x_n-a_i)+a_i$. Thus, the index $i_n$ is defined by $i_n=\arg\max\left\{0,||x_n-a_i||-r_i:i=1,\ldots,N\right\}.$ 
Since $C_n, Q_n$ are the half-spaces, $x_{n+1}$ is expressed by the explicit formula in \cite{CH2005,SS2000}. The parameters are 
$\alpha_n=\beta_n=\kappa=0$ and $\lambda=1$.

The first experiment is performed with $N=10^3, r_i=1$,
$$ a_i=\left(\frac{1}{2}\cos\frac{i\pi}{N}\sin\frac{2i\pi}{N};\frac{1}{2}\cos\frac{i\pi}{N}\cos\frac{2i\pi}{N};\frac{1}{2} \sin\frac{i\pi}{N}\right),$$
and the starting point $x_0=(1;2;7)$. Since $||a_i||=1/2<1=r_i $, $0\in K_i$ for all $i=1,\ldots,N$. So, $K=\cap_{i=1}^N K_i\ne \O$. In this experiment, 
the exact projection $P_K(x_0)$ of $x_0$ onto the feasibility set $K$ is unknown. For the fixed numbers of the iterations $n_{\max}$, 
Table $\ref{tab:1}$ gives time for PHM's execution in both parallel mode $(T_p)$ by using two processors and sequential mode $(T_s)$.

\begin{table}[ht]\caption{Results for the starting point $x_0=(1 ;2 ;7)$ and fixed numbers of iterations.}\label{tab:1}
\medskip\begin{center}
\begin{tabular}{|c|c|c|c|}
\hline
 $\qquad n_{\max} \qquad$ & \multicolumn{2}{c|}{PHM} &{$x_n$}
\\ \cline{2-3}
  & $\qquad T_p\qquad$ & $\qquad T_s\qquad$ & $\qquad \qquad$\\ \hline
 $250$ & 0.25 & 0.39 &  $(0.0699;0.0226;0.4726)$ \\
 $500$ & 0.43 & 0.79 &  $(0.0575  ;  0.0168  ;  0.3853)$ \\
  $1000$ & 0.89 & 1.60 &  $(0.0473    ;0.0119    ;0.3130)$ \\ 
$2000$ & 1.94 & 3.08 &  $(0.0388    ;0.0082    ;0.2533)$ \\ 
$5000$ & 4.05 & 7.70 &  $(0.0299    ;0.0048    ;0.1904)$ \\ \hline
 \end{tabular}\end{center}
\end{table}

The second experiment is performed with $N=10^3, r_i=1$,
$$ a_i=\left(\cos\frac{i\pi}{N}\sin\frac{2i\pi}{N};\cos\frac{i\pi}{N}\cos\frac{2i\pi}{N};\sin\frac{i\pi}{N}\right),$$
and the starting point $x_0=(-3;-5;-9)$. From $||a_i||=1=r_i $ for all $i=1,\ldots,N$, we see that the feasibility set $K=\cap_{i=1}^N K_i$ 
has the unique point $0$. So, $x^\dagger:=P_Kx_0=0$. For given tolerances TOLs, Table $\ref{tab:2}$ gives time for PHM' execution 
in both parallel mode (two processors) and sequential mode. Moreover, the sequence $\left\{x_n\right\}$ converges very quickly to $P_Kx_0$=0, 
and so our algorithm is effective.

\begin{table}[ht]\caption{Results for the starting point $x_0=(-3;-5;-9)$ and given tolerances.}\label{tab:2}
\medskip\begin{center}
\begin{tabular}{|c|c|c|c|c|}
\hline
  TOL& \multicolumn{2}{c|}{PHM} &{$x_n$}&$n_{\max}$
\\ \cline{2-3}
  & $\qquad T_p\qquad$ & $\qquad T_s\qquad$ & $\qquad \qquad$&\\ \hline
0.02500 & 0.27 & 0.51 &  $(-0.0056;       -0.0133;       -0.0212)$& $285$ \\
0.00750  & 1.04 & 1.88 &  $(-0.0024;      -0.0034;      -0.0063)$&$1088$ \\
0.00500   & 1.61 & 2.83 &  $(-0.0009      -0.0032      -0.0041)$&$1645$ \\ 
0.00150& 5.95 & 10.10 &  $(-0.0003     -0.0007      -0.0013)$& $5999$ \\ 
0.00075 & 11.70 & 21.39 &  $(-0.0001     -0.0004      -0.0005)$&$12178$ \\
 0.00015 & 53.85 & 103.97 &  $(-0.0000     -0.0001      -0.0001)$&$59416$ \\\hline
 \end{tabular}\end{center}
\end{table}
\textit{Example 2.} We consider the problem of finding a common fixed point of a finite family of mappings $\left\{S_j\right\}_{j=1}^M$. Let $H$ 
be the functional space $L^2[0,1]$, and $S_j:H\to H $ is defined by
$$
\left[S_j(x)\right](t):=\int_0^1 K_j(t,s)f_j(x(s))ds+g_j(t), j=1,2,3,4,
$$
where 
\begin{eqnarray*}
&&K_1(t,s)=\frac{2tse^{t+s}}{e\sqrt{e^2-1}}, f_1(x)=\cos x, g_1(t)=-\frac{2te^t}{e\sqrt{e^2-1}},\\ 
&&K_2(t,s)=\sqrt{3}ts, f_2(x)=\frac{1}{x^2+1}, g_2(t)=-\frac{\sqrt{3}}{2}t,\\
&&K_3(t,s)=\frac{\sqrt{21}}{7}|t-s|, f_3(x)=\sin x, g_3(t)=0,\\
&&K_4(t,s)=\frac{\sqrt{21}}{7}(t+s), f_4(x)=\exp(-x^2), g_4(t)=-\frac{\sqrt{21}}{7}\left(t+\frac{1}{2}\right).
\end{eqnarray*}
A straightforward computation ensures that $|f_j^{'}(x)|\le 1$ for all $x\in H$. Moreover, according to \cite{VGK1972}, 
the mappings $S_j$ are Frechet differentiable and $||S_j^{'}(x)h||\le ||h||$ for all $x,h\in H$. Hence, $||S_j^{'}(x)||\le 1$ for all $x\in H$. 
This implies that the mappings $S_j$ are $0$ - strict pseudocontractive on $H$. Besides, $x^\dagger=0$ 
is a common fixed point of the mappings $S_j, j=1,2,3,4$. In this example, we consider $A_i(x)=0$ for all $x\in H$ and 
$K_i=B[0,1]$ is a closed unit ball centered at the origin for each $i=1,\ldots,N$. Arccording to Algorithm 
$\ref{Algor.2}$, the intermediate aproximation $\bar{z}_n=y_n^i=z_n^i=x_n$ if $||x_n||\le 1$. Otherwise, 
$\bar{z}_n=y_n^i=z_n^i=x_n/||x_n||$. We chose $\beta_n=\kappa=0$ and calculate the approximations
\begin{equation}\label{eq:ex2.1}
\left\{
\begin{array}{ll}
&u_n^1(t)=\alpha_n x_n(t)+(1-\alpha_n)\frac{2te^{t}}{e\sqrt{e^2-1}}\left\{\int_0^1 se^s\cos \bar{z}_n(s)ds-1\right\},\\ 
& u_n^2(t)=\alpha_n x_n(t)+(1-\alpha_n)\sqrt{3}t\left\{\int_0^1 \frac{sds}{1+\bar{z}_n^2(s)}-\frac{1}{2}\right\},\\ 
& u_n^3(t)=\alpha_n x_n(t)+(1-\alpha_n)\frac{\sqrt{21}}{7}\int_0^1 |t-s|\sin \bar{z}_n(s)ds,\\
& u_n^4(t)=\alpha_n x_n(t)+(1-\alpha_n)\frac{\sqrt{21}}{7}\left\{\int_0^1 (t+s)e^{-\bar{z}_n^2(s)}ds-t-\frac{1}{2}\right\}.
\end{array}
\right.
\end{equation}
Thus, the furthest element $\bar{u}_n(t)$ from $x_n(t)$ among all $u_n^j(t)$ is chosen. The next iterate $x_{n+1}$ is also computed 
by the explicit formula in \cite{CH2005,SS2000}.

\begin{table}[ht]\caption{Experiment with the starting point $x_0(t)=1.$}\label{tab:4}
\medskip\begin{center}
\begin{tabular}{|c|c|c|c|}
\hline
 $\qquad n_{\max} \qquad$ & \multicolumn{2}{c|}{PHM} &{$TOL$}
\\ \cline{2-3}
  & $\qquad T_p\qquad$ & $\qquad T_s\qquad$ & $\qquad \qquad$\\ \hline
 $5$ & 6.85 & 12.97 &  $0.06427$ \\
$10$ & 25.13 & 49.44 &  $0.01042$ \\
$15$ & 64.56 & 123.12 &  $0.00090$ \\
$20$ & 121.32 & 233.84 &  $0.00056$ \\ \hline
 \end{tabular}\end{center}
\end{table}

\begin{table}[ht]\caption{Experiment with the starting point $x_0(t)=\frac{1}{100}e^{-10t}\sin(1000t).$}\label{tab:5}
\medskip\begin{center}
\begin{tabular}{|c|c|c|c|}
\hline
 $\qquad n_{\max} \qquad$ & \multicolumn{2}{c|}{PHM} &{$TOL$}
\\ \cline{2-3}
  & $\qquad T_p\qquad$ & $\qquad T_s\qquad$ & $\qquad \qquad$\\ \hline
 $5$ & 5.13 & 9.99 &  $0.00322$ \\
$10$ & 13.95 & 26.51 &  $0.00050$ \\
$15$ & 29.01 & 57.36 &  $0.00035$ \\
$20$ & 57.12 & 112.44 &  $0.00025$ \\ \hline
 \end{tabular}\end{center}
\end{table}
All programs are written in the C programming language. They are performed on the computing cluster LINUX IBM 1350 with 8 computing nodes. 
Each node contains two Intel Xeon dual core 3.2 GHz, 2GBRam. We use the following notations:

\begin{center}
\begin{tabular}{l l}
PHM & The parallel hybrid method\\
$TOL$ & Tolerance $\|x_n - x^\dagger\|$ \\
$T_p$  & Execution time of PHM in parallel mode (2CPUs - in seconds)\\
$T_s$ & Execution time of PHM in sequential mode (in seconds) \\
\end{tabular}
\end{center}

\begin{center}
\begin{figure}\sidecaption
\resizebox{0.65\hsize}{!}{\includegraphics{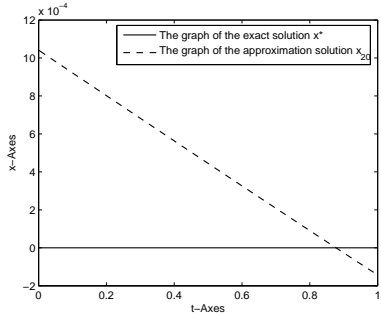}}
\caption{Geometric illustration of the exact solution $x^*(t)=0$ and the approximation solution $x_{20}(t)$ with the starting point $x_0(t)=1.$}
\label{fig:01}
\end{figure}
\end{center}

\begin{center}
\begin{figure}\sidecaption
\resizebox{0.50\hsize}{!}{\includegraphics{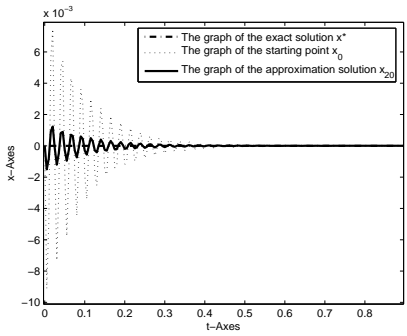}}
\caption{Geometric illustration of the exact solution $x^*(t)=0$ and the approximation solution $x_{20}(t)$ with the starting point $x_0(t)=1/100\exp(-10t)\sin(1000t)$.}
\label{fig:02}
\end{figure}
\end{center}

All integrals in Example 2  are calculated by 
using the trapezoidal formula with the stepsize $\tau=0.001$. In the next two experiments, we chose 
$\beta_n=\kappa=0$ and $\alpha_n=\frac{1}{n+1}$. The former is performed with the starting point $x_0(t)=1$ and the latter is with 
$x_0(t)=\frac{1}{100}e^{-10t}\sin(1000t)$. For the fixed numbers of iterations $n_{\max}$, Tables $\ref{tab:4}$ and $\ref{tab:5}$ 
give execution time of PHM in parallel mode ($T_p$) by using two processors and sequential mode ($T_s$). The last column are the 
obtained tolerances which are the distances from the approximation solutions to the exact solution $x^\dagger$. Figures $\ref{fig:01}$ 
and $\ref{fig:02}$ illustrate the graphs of the the starting point $x_0(t)$, the approximation solution $x_{20}(t)$ and the exact solution 
$x^\dagger(t)=0$. From our numerical experiments, we see that the maximal speed-up of the proposed parallel algorithm is 
$S_p=T_s/T_p \approx 2.0$. So, the efficiency of the parallel computation by using two processors is $E_p=S_p/2\approx 1.0$.
%

\end{document}